\documentclass[reqno]{amsart}
\usepackage{amssymb,mathtools,lineno}

\usepackage[colorlinks=true,citecolor=blue,linkcolor=green,urlcolor=red]{hyperref}

\allowdisplaybreaks[4]

\theoremstyle{plain}
\newtheorem{thm}{Theorem}[section]
\newtheorem{cor}{Corollary}[section]
\newtheorem{lem}{Lemma}[section]

\newtheorem*{thma}{Theorem~A}
\newtheorem*{thmb}{Theorem~B}
\newtheorem*{thmc}{Theorem~C}
\newtheorem*{thmd}{Theorem~D}
\newtheorem*{thme}{Theorem~E}

\theoremstyle{remark}
\newtheorem{rem}{Remark}[section]

\DeclareMathOperator{\td}{d\!}
\DeclareMathOperator{\bell}{B}

\numberwithin{equation}{section}

\begin{document}

\title[Identities derived from Gauss hypergeometric functions]
{Combinatorial identities derived from explicit formulas of Gauss hypergeometric functions}

\author[F. Qi]{Feng Qi*}
\address{School of Mathematics and Physics, Hulunbuir University, Inner Mongolia, 021008, China;
17709 Sabal Court, University Village, Dallas, TX 75252-8024, USA}
\email{qifeng618@gmail.com}
\urladdr{\url{https://orcid.org/0000-0001-6239-2968}}

\begin{abstract}
In the present paper, with the help of the Fa\`a di Bruno formula and several identities of partial Bell polynomials, the author establishes explicit formulas of the Gauss hypergeometric functions
\begin{gather*}
{\,}_2F_1\biggl(\frac{1-n}{2},\frac{2-n}{2};\frac{3}{2}-m;z^2\biggr), \quad
{\,}_2F_1\biggl(-\frac{n}{2},\frac{1-n}{2};\frac{1}{2}-m;z^2\biggr),\\
{\,}_2F_1\biggl(a,a+\frac{1}{2};\frac{3}{2}-m;z^2\biggr), \quad
{\,}_2F_1\biggl(a,a+\frac{1}{2};\frac{1}{2}-m;z^2\biggr)
\end{gather*}
for $m,n\in\mathbb{N}$ and $a\in\mathbb{C}$, and then derives two combinatorial identities
\begin{equation*}
\sum_{k=0}^{m}\frac{2^k}{k!} \binom{2m-2k}{m-k}
\sum_{\ell=0}^{k} \frac{(-1)^\ell}{2^\ell}
\frac{(2k-2\ell-1)!!}{(n-\ell)!}
\binom{2k-\ell-1}{\ell-1}
=\frac{1}{n!}\binom{2m-n}{m}
\end{equation*}
and
\begin{equation*}
\sum_{k=1}^{m}\frac{1}{(k!)^2}\binom{2m-2k}{m-k}
\sum_{\ell=1}^{k} \binom{k}{\ell}\ell(2k-\ell-1)! (2a)_\ell
=\binom{2m+2a}{m}-\binom{2m}{m},
\end{equation*}
where $m\in\mathbb{N}_0$, $n\in\mathbb{Z}$, and $a\in\mathbb{C}$. These newly-established identities extend and generalize the beautiful combinatorial identity
\begin{equation*}
\sum_{k=0}^{n} \frac{2^{k}}{k!}\binom{2n-2k}{n-k}
\sum_{j=0}^{k}\frac{(-1)^{j}}{2^j}
\frac{(2k-2j-1)!!}{(n-j)!}
\binom{2k-j-1}{j-1}
=\frac{1}{n!}, \quad n\in\mathbb{N}_0,
\end{equation*}
which was obtained in Theorem~4 of the paper ``F. Qi, C.-Y. He, and D. Lim, \textit{Explicit formulas of two Gauss hypergeometric functions and several combinatorial identities}, Discrete Appl. Math. \textbf{393} (2026), 215\nobreakdash--229. DOI: \url{https://doi.org/10.1016/j.dam.2026.06.023}''. 
\end{abstract}



\thanks{*Corresponding author: Feng Qi, qifeng618@gmail.com}

\thanks{This paper was typeset using \AmS-\LaTeX}

\maketitle

\section{Preliminaries}
To ensure everything proceeds smoothly in this paper, we need to prepare the essential preliminaries.

\subsection{Basic notations}
The classical Euler gamma function $\Gamma(z)$ can be defined~\cite[Chapter~3]{Temme-96-book} by
\begin{equation*}
\Gamma(z)=\lim_{n\to\infty}\frac{n!n^z}{\prod_{k=0}^n(z+k)},
\quad
z\in\mathbb{C}\setminus\{0,-1,-2,\dotsc\}.
\end{equation*}
\par
For $n\in\mathbb{N}_0=\{0,1,2,\dotsc\}$ and $z\in\mathbb{C}$, the falling factorial $\langle z\rangle_n$ and the rising factorial $(z)_n$ (also known as the Pochhammer symbol or shifted factorial) are defined, respectively, by
\begin{equation*}
\langle z\rangle_n
=\prod_{k=0}^{n-1}(z-k)
=
\begin{cases}
z(z-1)\dotsm(z-n+1), & n\in\mathbb{N};\\
1, & n=0
\end{cases}
\end{equation*}
and
\begin{equation*}
(z)_n
=\prod_{\ell=0}^{n-1}(z+\ell)
=
\begin{cases}
z(z+1)\dotsm(z+n-1), & n\in\mathbb{N};\\
1, & n=0,
\end{cases}
\end{equation*}
where $\mathbb{N}=\{1,2,\dotsc\}$.
Both the falling factorial $\langle z\rangle_n$ and the rising factorial $(z)_n$ admit extensions to arbitrary integers $n\in\mathbb{Z}=\{0,\pm1,\pm2,\dotsc\}$ via the gamma function $\Gamma(z)$, namely,
\begin{equation*}
\langle z\rangle_n=\frac{\Gamma(z+1)}{\Gamma(z-n+1)}
\quad\text{and}\quad
(z)_n=\frac{\Gamma(z+n)}{\Gamma(z)}.
\end{equation*}
\par
In~\cite[p.~154]{GKP-Concrete-Math-2nd} and~\cite[p.~17, Identity~17]{Spivey-art-2019}, the classical binomial coefficient $\binom{n}{k}$ for $n,k\in\mathbb{N}_0$ with $n\ge k$ is generalized as
\begin{equation}\label{falling-binomial-eq}
\binom{z}{k}=
\begin{dcases}
\frac{\langle z\rangle_k}{k!}, & k\ge0;\\
0, & k<0
\end{dcases}
\end{equation}
for $z\in\mathbb{R}$ and $k\in\mathbb{Z}$. According to the definition in~\eqref{falling-binomial-eq}, we have $\binom{-1}{-1}=0$, which contradicts the convention $\binom{-1}{-1}=1$ adopted in~\cite[Theorem~4]{DA19034-cas-sc.tex}. To avoid this inconsistency, we redefine the extended binomial coefficient $\binom{z}{w}$ for $z,w\in\mathbb{C}$ as
\begin{equation*}
\binom{z}{w}=
\begin{dcases}
\frac{\Gamma(z+1)}{\Gamma(w+1)\Gamma(z-w+1)}, & z\not\in\mathbb{N}_-,\quad w,z-w\not\in\mathbb{N}_-;\\
0, & z\not\in\mathbb{N}_-,\quad w\in\mathbb{N}_- \text{ or } z-w\in\mathbb{N}_-;\\
\frac{\langle z\rangle_w}{w!},& z\in\mathbb{N}_-, \quad w\in\mathbb{N}_0;\\
\frac{\langle z\rangle_{z-w}}{(z-w)!}, & z,w\in\mathbb{N}_-, \quad z-w\in\mathbb{N}_0;\\
0, & z,w\in\mathbb{N}_-, \quad z-w\in\mathbb{N}_-;\\
\infty, & z\in\mathbb{N}_-, \quad w\not\in\mathbb{Z},
\end{dcases}
\end{equation*}
such that $\binom{-1}{-1}=1$ is valid, where $\mathbb{N}_-=\{-1,-2,\dotsc\}$ denotes the set of all negative integers. In what follows, we always use this definition by default.
\par
The double factorial of negative odd integers $-(2k+1)$ is defined by
\begin{equation*}
(-2k-1)!!=\frac{(-1)^k}{(2k-1)!!}=(-1)^k\frac{(2k)!!}{(2k)!}, \quad k\in\mathbb{N}_0.
\end{equation*}

\subsection{Gauss hypergeometric function and Chu--Vandermonde identity}
For $a,b\in\mathbb{C}$, $c\in\mathbb{C}\setminus\{0,-1,-2,\dotsc\}$, and $z\in\mathbb{C}$, the Gauss hypergeometric function ${\,}_2F_1(a,b;c;z)$ is defined by
\begin{equation}\label{Gauss-hypergeom-f}
{\,}_2F_1(a,b;c;z)
=\sum_{n=0}^\infty
\frac{(a)_n(b)_n}{(c)_n}\frac{z^n}{n!}.
\end{equation}
If $a\in\{0,-1,-2,\dotsc\}$ or $b\in\{0,-1,-2,\dotsc\}$, then the series in~\eqref{Gauss-hypergeom-f} terminates and hence reduces to a polynomial in $z$. Otherwise, when $a,b\notin\{0,-1,-2,\dotsc\}$, the series converges for $|z|<1$ and, on the unit circle $|z|=1$, is
\begin{enumerate}
\item
divergent if $\Re(a+b-c)\ge1$;
\item
absolutely convergent if $\Re(a+b-c)<0$;
\item
conditionally convergent if $0\le\Re(a+b-c)<1$, except at the point $z=1$.
\end{enumerate}
When $c\in\{0,-1,-2,\dotsc\}$, the function ${\,}_2F_1(a,b;c;z)$ is still well defined provided that $a\in\{-1,-2,\dotsc\}$ with $a>c$, or $b\in\{-1,-2,\dotsc\}$ with $b>c$. Moreover, the Gauss hypergeometric function ${\,}_2F_1(a,b;c;z)$ is single-valued analytic on $\mathbb{C}\setminus[1,\infty)$. For further details on ${\,}_2F_1(a,b;c;z)$, see~\cite[Chapter~5]{Temme-96-book}.
\par
The well-known Chu--Vandermonde identity
\begin{equation}\label{Entry15.4.24NIST-HB-2010}
{\,}_2F_1(-n,b;c;1)=\frac{(c-b)_n}{(c)_n},\quad n\in\mathbb{N}_0
\end{equation}
in~\cite[p.~387, Entry~15.4.24]{NIST-HB-2010} will be used in what follows.

\subsection{Partial Bell polynomials and Fa\`a di Bruno formula}
The partial Bell polynomials $\bell_{n,k}(z_1,z_2,\dotsc,z_{n-k+1})$, also known as the Bell polynomials of the second kind, are defined for $(z_1,z_2,\dotsc,z_{n-k+1})\in\mathbb{C}^{\,n-k+1}$ and $n,k\in\mathbb{N}_0$ with $n\ge k$ by
\begin{equation*}
\frac{\bell_{n,k}(z_1,z_2,\dotsc,z_{n-k+1})}{n!}
=
\sum_{\substack{\sum_{i=1}^{n-k+1}i\ell_i=n,\\ \sum_{i=1}^{n-k+1}\ell_i=k,\,\ell_i\in\mathbb{N}_0}}
\prod_{i=1}^{n-k+1}
\biggl[
\frac{1}{\ell_i!}
\biggl(\frac{z_i}{i!}\biggr)^{\ell_i}
\biggr],
\end{equation*}
with the special cases $\bell_{0,0}(z_1)=1$ and
\begin{equation*}
\bell_{k,0}(z_1,z_2,\dotsc,z_{k+1})=0,\quad k\in\mathbb{N};
\end{equation*}
see~\cite[Definition~11.2]{Charalambides-book-2002} and \cite[p.~134, Theorem~A]{Comtet-Combinatorics-74}.
They satisfy the scaling identity
\begin{equation}\label{Bell(n-k)}
\bell_{n,k}\bigl(abz_1,ab^2z_2,\dotsc,ab^{n-k+1}z_{n-k+1}\bigr)
=
a^kb^n
\bell_{n,k}(z_1,z_2,\dotsc,z_{n-k+1}),
\end{equation}
where $n,k\in\mathbb{N}_0$ with $n\ge k$; see~\cite[p.~412]{Charalambides-book-2002} and~\cite[p.~135]{Comtet-Combinatorics-74}.
\par
On~\cite[p.~169]{CDM-68111.tex}, in the proof of~\cite[Theorem~3.2]{CDM-68111.tex}, Qi and his coauthors established the formula
\begin{multline}\label{Bell-Polyn-Half}
\bell_{n,k}\biggl(\biggl\langle\frac12\biggr\rangle_1, \biggl\langle\frac12\biggr\rangle_2, \dotsc, \biggl\langle\frac12\biggr\rangle_{n-k+1}\biggr)\\
=(-1)^{n+k}\frac{(2n-2k-1)!!}{2^n}\binom{2n-k-1}{k-1}
\end{multline}
for $n,k\in\mathbb{N}_0$ with $n\ge k$. In what follows, this formula will play an important role technically.
\par
Expressed in terms of the partial Bell polynomials $\bell_{n,k}$, the famous Fa\`a di Bruno formula takes the form
\begin{equation}\label{Bruno-Bell-Polynomial}
[f\circ h(z)]^{(n)}
=
\sum_{k=0}^n
f^{(k)}(h(z))
\bell_{n,k}\bigl(
h'(z),h''(z),\dotsc,h^{(n-k+1)}(z)
\bigr),
\end{equation}
where $n\in\mathbb{N}_0$, $f$ and $h$ are $n$ times differentiable, and $f\circ h$ denotes the composition of $f$ and $h$; see~\cite[Theorem~11.4]{Charalambides-book-2002} and~\cite[p.~139, Theorem~C]{Comtet-Combinatorics-74}.

\subsection{Two lemmas}
To facilitate the proofs of the main results presented in the next section, we now establish the following lemmas.

\begin{lem}\label{lem-deriv-Gauss}
For $\alpha,\beta,\gamma\in\mathbb{C}$ and $n\in\mathbb{N}_0$, we have the derivative formula
\begin{multline}\label{Eq-deriv-Gauss}
\bigl[\bigl(\alpha+\beta\sqrt{z}\,\bigr)^{\gamma}\bigr]^{(n)}\\
=\frac{(-1)^n}{(2z)^n} \sum_{k=0}^{n}(-\gamma)_k (2n-2k-1)!!\binom{2n-k-1}{k-1}\bigl(\beta\sqrt{z}\,\bigr)^k \bigl(\alpha+\beta\sqrt{z}\,\bigr)^{\gamma-k}.
\end{multline}
\end{lem}

\begin{proof}
In view of the Fa\`a di Bruno formula~\eqref{Bruno-Bell-Polynomial}, with the help of the identity~\eqref{Bell(n-k)}, and by virtue of the formula~\eqref{Bell-Polyn-Half}, we arrive at
\begin{align*}
&\bigl[\bigl(\alpha+\beta\sqrt{z}\,\bigr)^{\gamma}\bigr]^{(n)}
=\sum_{k=0}^{n}(u^\gamma)^{(k)} \bell_{n,k}\biggl(\beta\biggl\langle\frac{1}{2}\biggr\rangle_1z^{1/2-1}, \beta\biggl\langle\frac{1}{2}\biggr\rangle_2z^{1/2-2},\\
&\quad\dotsc,\beta\biggl\langle\frac{1}{2}\biggr\rangle_{n-k+1}z^{1/2-(n-k+1)}\biggr), \quad u=u(z)=\alpha+\beta\sqrt{z}\,,\\
&=\sum_{k=0}^{n}\langle\gamma\rangle_k\bigl(\alpha+\beta\sqrt{z}\,\bigr)^{\gamma-k} \beta^kz^{k/2-n} \bell_{n,k}\biggl(\biggl\langle\frac12\biggr\rangle_1, \biggl\langle\frac12\biggr\rangle_2,\dotsc, \biggl\langle\frac12\biggr\rangle_{n-k+1}\biggr)\\
&=\sum_{k=0}^{n}\langle\gamma\rangle_k\bigl(\alpha+\beta\sqrt{z}\,\bigr)^{\gamma-k} \beta^kz^{k/2-n} (-1)^{n+k}\frac{(2n-2k-1)!!}{2^n}\binom{2n-k-1}{k-1}\\
&=\frac{(-1)^n}{(2z)^n} \sum_{k=0}^{n}(-\gamma)_k (2n-2k-1)!!\binom{2n-k-1}{k-1}\bigl(\beta\sqrt{z}\,\bigr)^k \bigl(\alpha+\beta\sqrt{z}\,\bigr)^{\gamma-k}.
\end{align*}
The proof of Lemma~\ref{lem-deriv-Gauss} is complete.
\end{proof}

\begin{lem}\label{Lem2-ID}
For $n\in\mathbb{N}$ and $a\in\mathbb{R}$, we have
\begin{equation}\label{ID-Lem2}
\sum_{k=0}^{n}(2n-2k-1)!!\binom{2n-k-1}{k-1}\frac{(2a-1)_{k}}{2^{k-1}}
=\frac{(2a-1) \Gamma(2a+2n-1)}{2^{n-1}\Gamma(2a+n)}.
\end{equation}
\end{lem}

\begin{proof}
A straightforward computation gives
\begin{align*}
&\quad\sum_{k=0}^{n}(2n-2k-1)!!\binom{2n-k-1}{k-1}\frac{(2a-1)_{k}}{2^{k-1}}\\
&=\frac{1}{2^{n-1}}\sum_{k=1}^{n}\frac{(2n-k-1)!}{(n-k)!(k-1)!}(2a-1)_{k}\\
&=\frac{1}{2^{n-1}}\sum_{k=0}^{n-1}\frac{(2n-k-2)!}{(n-k-1)!k!}(2a-1)_{k+1}\\
&=\frac{2a-1}{2^{n-1}}\sum_{k=0}^{n-1}\frac{(2n-k-2)!}{(n-k-1)!k!}(2a)_{k}\\
&=\frac{(2a-1)(2n-2)!}{2^{n-1}(n-1)!}\sum_{k=0}^{n-1}\frac{(2a)_{k}(1-n)_k}{(2-2n)_k}\frac{1}{k!}\\
&=\frac{(2a-1)(2n-2)!}{2^{n-1}(n-1)!}{\,}_2F_1(2a,1-n;2-2n;1),
\end{align*}
where, by virtue of the Chu--Vandermonde identity~\eqref{Entry15.4.24NIST-HB-2010},
\begin{equation*}
{\,}_2F_1(2a,1-n;2-2n;1)=\frac{(2-2n-2a)_{n-1}}{(2-2n)_{n-1}}
=\frac{(n-1)!}{(2n-2)!}\frac{\Gamma(2a+2n-1)}{\Gamma(2a+n)}
\end{equation*}
for $n\in\mathbb{N}$. Consequently, the identity~\eqref{ID-Lem2} is thus proved. The proof of Lemma~\ref{Lem2-ID} is complete.
\end{proof}

\section{Motivations and main results}
In~\cite{DA19034-cas-sc.tex}, among other findings, the following main conclusions were established.

\begin{thma}[{\cite[Theorem~1]{DA19034-cas-sc.tex}}]
For $n\in\mathbb{N}$, the Gauss hypergeometric function ${\,}_2F_1$ has the explicit expression
\begin{multline}\label{2F1(izan-peraz)}
{\,}_2F_1\biggl(\frac{1-n}{2},\frac{2-n}{2};\frac{3}{2}-n;z^2\biggr)
=\frac{1}{4}\frac{1}{\binom{2n-2}{n-1}}\sum_{k=1}^{n}2^k\binom{2n-k-1}{n-1}\binom{n-1}{k-1}\\
\times\bigl[(1-z)^{n-k}+(-1)^{k-1}(1+z)^{n-k}\bigr]z^{k-1}.
\end{multline}
\end{thma}

\begin{thmb}[{\cite[Theorem~2]{DA19034-cas-sc.tex}}]
For $n\in\mathbb{N}$, the Gauss hypergeometric function ${\,}_2F_1$ has the explicit expression
\begin{multline}\label{2F1(izan-peraz)-more}
{\,}_2F_1\biggl(-\frac{n}{2},\frac{1-n}{2};\frac{1}{2}-n;z^2\biggr)
=\frac{1}{2}\frac{n!}{\binom{2n}{n}}\sum_{k=0}^{n} \frac{2^{k}}{k!}\binom{2n-2k}{n-k}\sum_{j=0}^{k}\frac{(2k-2j-1)!!}{(n-j)!}\\
\times\binom{2k-j-1}{j-1}\bigl[(1-z)^{n-j}+(-1)^{j}(1+z)^{n-j}\bigr]z^{j}.
\end{multline}
\end{thmb}

In light of the explicit formula~\eqref{2F1(izan-peraz)} in Theorem~A and the explicit formula~\eqref{2F1(izan-peraz)-more} in Theorem~B, the following two combinatorial identities were discovered.

\begin{thmc}[{\cite[p.~224, Eq.~(37)]{DA19034-cas-sc.tex}}]
For $j,n\in\mathbb{N}_0$, we have
\begin{equation}\label{Not-Found-ID-Spivey-Riodan-Gen}
\sum_{k=0}^n 2^{2k}\binom{n}{k} \binom{k}{j-k}=2^j\binom{2n}{j}.
\end{equation}
\end{thmc}

\begin{thmd}[{\cite[Theorem~4]{DA19034-cas-sc.tex}}]
For $n\in\mathbb{N}_0$, the combinatorial identity
\begin{equation}\label{ID-n-Factorial2sums}
\sum_{k=0}^{n} \frac{2^{k}}{k!}\binom{2n-2k}{n-k} \sum_{j=0}^{k}\frac{(-1)^{j}}{2^j}\frac{(2k-2j-1)!!}{(n-j)!}\binom{2k-j-1}{j-1}
=\frac{1}{n!}
\end{equation}
is valid with the convention $\binom{-1}{-1}=1$.
\end{thmd}

\begin{rem}
In~\cite[Remark~13]{DA19034-cas-sc.tex}, it was remarked that the combinatorial identity~\eqref{Not-Found-ID-Spivey-Riodan-Gen} recovers Eq.~(1.64) on Page~9 in ``\emph{Tables of Combinatorial Identities}, Vol.~4, Eight tables based on seven unpublished manuscript notebooks (1945\nobreakdash--1990) of H. W. Gould, which include series techniques and certain special numbers. Edited and Compiled by Prof. Jocelyn Quaintance, May 2010''. URL: \url{https://math.wvu.edu/\%7Ehgould/Vol.4.PDF}.
\end{rem}

\begin{rem}
In~\cite[Theorem~5.1]{Special-Bell2Euler.tex}, Qi and Guo discovered the formula
\begin{equation}\label{Bell-x-1-0-eq}
\bell_{n,k}(z,1,0,\dotsc,0)
=\frac{(n-k)!}{2^{n-k}}\binom{n}{k}\binom{k}{n-k}z^{2k-n}
\end{equation}
for $n,k\in\mathbb{N}_0$ with $n\ge k$.
In the proof of~\cite[Proposition~2]{DA19034-cas-sc.tex}, by virtue of the Fa\`a di Bruno formula~\eqref{Bruno-Bell-Polynomial} and the formula~\eqref{Bell-x-1-0-eq}, the explicit formula
\begin{equation*}
P_n(z)=\frac{\bigl(z^2-1\bigr)^{n}}{2^{2n}z^{n}}\sum_{k=0}^{n} \binom{n}{k}\binom{k}{n-k}2^{2k}\biggl(\frac{z^{2}}{z^2-1}\biggr)^k
\end{equation*}
for $n\in\mathbb{N}_0$ was derived, where 
\begin{equation*}
P_n(z)=\frac{1}{(2n)!!}\frac{\operatorname{d}^n}{\td z^n}\bigl[\bigl(z^2-1\bigr)^n\bigr], \quad n\in\mathbb{N}_0
\end{equation*}
denotes the Legendre polynomials; see~\cite[p.~675]{Brychkov-CRC-2008}.
\end{rem}

In~\cite[Theorem~5]{DA19034-cas-sc.tex}, the formula~\eqref{Bell-x-1-0-eq} was combinatorially extended as follows.

\begin{thme}[{\cite[Theorem~5]{DA19034-cas-sc.tex}}]
For $n,k\in\mathbb{N}_0$ with $n\ge k$, the Bell polynomials of the second kind $\bell_{n,k}$ satisfy
\begin{equation}\label{Bell-x-1-0-combin-proof-eq}
\bell_{n,k}(z_1,z_2,0,\dotsc,0)=\frac{(n-k)!}{2^{n-k}}\binom{n}{k}\binom{k}{n-k}z_1^{2k-n}z_2^{n-k}.
\end{equation}
\end{thme}

\begin{rem}
In view of the identity~\eqref{Bell(n-k)}, we observe that the formulas~\eqref{Bell-x-1-0-eq} and~\eqref{Bell-x-1-0-combin-proof-eq} are equivalent to each other. The distinction between them lies in their proofs: The formula~\eqref{Bell-x-1-0-eq} was analytically established in~\cite{Special-Bell2Euler.tex}, whereas the formula~\eqref{Bell-x-1-0-combin-proof-eq} was combinatorially obtained in~\cite{DA19034-cas-sc.tex}.
\end{rem}

In~\cite[Remark~14]{DA19034-cas-sc.tex}, the following question was asked:
For $m,n\in\mathbb{N}$, what are the explicit expressions of the Gauss hypergeometric functions
\begin{equation}\label{2Gauss}
{\,}_2F_1\biggl(-\frac{n}{2},\frac{1-n}{2};\frac{1}{2}-m;z^2\biggr)
\quad\text{and}\quad
{\,}_2F_1\biggl(\frac{1-n}{2},\frac{2-n}{2};\frac{3}{2}-m;z^2\biggr)?
\end{equation}
\par
Since, for every $n\in\mathbb{N}$, one of $-\frac{n}{2}\in\{0,-1,-2,\dotsc\}$ and $\frac{1-n}{2}\in\{0,-1,-2,\dotsc\}$ holds, and likewise one of $\frac{1-n}{2}\in\{0,-1,-2,\dotsc\}$ and $\frac{2-n}{2}\in\{0,-1,-2,\dotsc\}$ holds, both Gauss hypergeometric functions in~\eqref{2Gauss} are polynomials in $z$.
\par
In this paper, we aim to address the question posed above.
\par
The main results of this paper are summarized as follows.

\begin{thm}\label{Gauaa1-thm}
For $m,n\in\mathbb{N}$, the Gauss hypergeometric function ${\,}_2F_1$ has the explicit expression
\begin{multline}\label{Gauaa1-eq}
{\,}_2F_1\biggl(\frac{1-n}{2},\frac{2-n}{2};\frac{3}{2}-m;z^2\biggr)
=\frac{1}{2^{m+1}}\frac{(m-1)!}{(2m-3)!!} \sum_{k=1}^{m}2^k\binom{n-1}{k-1}\\
\times\binom{2m-k-1}{m-1}\bigl[(1-z)^{n-k}-(-1)^{k}(1+z)^{n-k}\bigr]z^{k-1}.
\end{multline}
\end{thm}

\begin{thm}\label{Gauaa2-thm}
For $m\in\mathbb{N}_0$ and $n\in\mathbb{N}$, the Gauss hypergeometric function ${\,}_2F_1$ has the explicit expression
\begin{multline}\label{Gauaa2-eq}
{\,}_2F_1\biggl(-\frac{n}{2},\frac{1-n}{2};\frac{1}{2}-m;z^2\biggr)
=\frac{n!}{2\binom{2m}{m}} \sum_{k=0}^{m}\frac{2^k}{k!}\binom{2m-2k}{m-k} \\
\times \sum_{\ell=0}^{k} \frac{(2k-2\ell-1)!!}{(n-\ell)!} \binom{2k-\ell-1}{\ell-1} \bigl[(1-z)^{n-\ell} +(-1)^{\ell}(1+z)^{n-\ell}\bigr] z^\ell.
\end{multline}
\end{thm}

The preceding two theorems can be extended to yield the following pair of results.

\begin{thm}\label{comb-id-thm1}
For $m\in\mathbb{N}_0$ and $a\in\mathbb{C}\setminus\{\frac12\}$, the Gauss hypergeometric function ${\,}_2F_1$ has the explicit expression
\begin{multline}\label{comb-id-Eq1}
{\,}_2F_1\biggl(a,a+\frac{1}{2};\frac{3}{2}-m;z^2\biggr)
=\frac{1}{2(2a-1)} \frac{1}{(2m-3)!!} \sum_{k=0}^{m}(2m-2k-1)!!\binom{2m-k-1}{k-1}\\
\times\bigl[(2a-1)_k(1+z)^{1-2a-k} -\langle1-2a\rangle_k(1-z)^{1-2a-k}\bigr]z^{k-1}.
\end{multline}
\end{thm}

\begin{thm}\label{comb-id-thm2}
For $m\in\mathbb{N}_0$ and $a\in\mathbb{C}$, the Gauss hypergeometric function ${\,}_2F_1$ has the explicit expression
\begin{multline}\label{comb-id-Eq2}
{\,}_2F_1\biggl(a,a+\frac{1}{2};\frac{1}{2}-m;z^2\biggr)
=\frac{1}{2^{m+1}}\frac{m!}{(2m-1)!!} \sum_{k=0}^m\frac{2^k}{k!}\binom{2m-2k}{m-k} \sum_{\ell=0}^{k} (2k-2\ell-1)!!\\
\times \binom{2k-\ell-1}{\ell-1}\bigl[(2a)_\ell(1+z)^{-2a-\ell}+\langle-2a\rangle_\ell (1-z)^{-2a-\ell}\bigr]z^\ell.
\end{multline}
\end{thm}

\begin{rem}
Setting $a=\tfrac{1-n}{2}$ in Theorem~\ref{comb-id-thm1} yields Theorem~\ref{Gauaa1-thm}, whose specialization to $m=n$ is Theorem~A.
Similarly, setting $a=-\tfrac{n}{2}$ in Theorem~\ref{comb-id-thm2} yields Theorem~\ref{Gauaa2-thm}, whose specialization to $m=n$ is Theorem~B.
\end{rem}

In the next section, with the aid of Lemmas~\ref{lem-deriv-Gauss} and~\ref{Lem2-ID}, we will present detailed proofs of the four theorems above.
\par
In Section~\ref{sec-combin-IDs}, we will present several combinatorial identities, together with a corollary and some remarks, derived from the preceding theorems.

\section{Proofs of main results}
We are now in a position to establish our main results.

\begin{proof}[Proof of Theorem~\ref{Gauaa1-thm}]
In~\cite[pp.~1015--1016]{Gradshteyn-Ryzhik-Table-8th}, we find
\begin{align*}
{\,}_2F_1\biggl(\frac{1-n}{2},-\frac{n-2}{2};\frac{3}{2};\frac{z^2}{t^2}\biggr)&=\frac{(t+z)^n-(t-z)^n}{2n z t^{n-1}}
\intertext{and}
{\,}_2F_1\biggl(-\frac{n-2}{2},\frac{1-n}{2};\frac{3}{2};-\tan^2z\biggr)&=\frac{\sin(n z)}{n\sin z\cos^{n-1}z}
\end{align*}
for $n\in\mathbb{N}$. These two formulas can be rewritten as
\begin{equation}\label{gauss-form1}
{\,}_2F_1\biggl(\frac{1-n}{2},\frac{2-n}{2};\frac{3}{2}; z\biggr)=\frac{1}{2n} \frac{\bigl(1+\sqrt{z}\,\bigr)^n-\bigl(1-\sqrt{z}\,\bigr)^n}{\sqrt{z}\,}, \quad n\in\mathbb{N}.
\end{equation}
\par
In~\cite[p.~557, Entry~15.2.4]{abram}, we find the derivative formula
\begin{equation}\label{drive-form-gauss-n}
[z^{c-1}{\,}_2F_1(a,b;c;z)]^{(n)}=(c-n)_n z^{c-n-1}{\,}_2F_1(a,b;c-n;z), \quad n\in\mathbb{N}.
\end{equation}
Replacing $n$ by $m$ in~\eqref{drive-form-gauss-n}, taking $(a,b;c)=\bigl(\frac{1-n}{2},\frac{2-n}{2};\frac{3}{2}\bigr)$ in~\eqref{drive-form-gauss-n}, and using~\eqref{gauss-form1}, we arrive at
\begin{align*}
{\,}_2F_1\biggl(\frac{1-n}{2},\frac{2-n}{2};&\frac{3}{2}-m;z\biggr)
=\frac{z^{m-1/2}}{\bigl(\frac32-m\bigr)_m} \biggl[\sqrt{z}\,{\,}_2F_1\biggl(\frac{1-n}{2},\frac{2-n}{2};\frac{3}{2};z\biggr)\biggr]^{(m)}\\
&=\frac{z^{m-1/2}}{\bigl(\frac32-m\bigr)_m} \biggl[\frac{\bigl(1+\sqrt{z}\,\bigr)^n-\bigl(1-\sqrt{z}\,\bigr)^n}{2n}\biggr]^{(m)}\\
&=\frac1{2n}\frac{z^{m-1/2}}{\bigl(\frac32-m\bigr)_m} \Bigl\{\bigl[\bigl(1+\sqrt{z}\,\bigr)^n\bigr]^{(m)} -\bigl[\bigl(1-\sqrt{z}\,\bigr)^n\bigr]^{(m)}\Bigr\}
\end{align*}
for $m,n\in\mathbb{N}$, where, by taking $\alpha=1$, $\beta=\pm1$, and $\gamma=n$ in Lemma~\ref{lem-deriv-Gauss},
\begin{multline}\label{Plus-Deriv}
\bigl[\bigl(1+\sqrt{z}\,\bigr)^n\bigr]^{(m)}\\
=\frac{(-1)^m}{(2z)^m} \sum_{k=0}^{m}(-n)_k (2m-2k-1)!!\binom{2m-k-1}{k-1}\bigl(1+\sqrt{z}\,\bigr)^{n-k}\bigl(\sqrt{z}\,\bigr)^k
\end{multline}
and
\begin{multline}\label{Minus-Deriv}
\bigl[\bigl(1-\sqrt{z}\,\bigr)^n\bigr]^{(m)}\\
=\frac{(-1)^m}{(2z)^m} \sum_{k=0}^{m} \langle n\rangle_k(2m-2k-1)!!\binom{2m-k-1}{k-1} \bigl(1-\sqrt{z}\,\bigr)^{n-k}\bigl(\sqrt{z}\,\bigr)^k
\end{multline}
for $m,n\in\mathbb{N}_0$. As a result, we arrive at
\begin{align*}
{\,}_2F_1\biggl(\frac{1-n}{2},&\frac{2-n}{2};\frac{3}{2}-m;z\biggr)
=\frac{1}{n}\frac{(-1)^{m+1}}{2^{m+1}}\frac{1}{\bigl(\frac32-m\bigr)_m} \sum_{k=0}^{m}\langle n\rangle_k (2m-2k-1)!! \\
&\quad\times\binom{2m-k-1}{k-1}\bigl[\bigl(1-\sqrt{z}\,\bigr)^{n-k}-(-1)^{k}\bigl(1+\sqrt{z}\,\bigr)^{n-k}\bigr] \bigl(\sqrt{z}\,\bigr)^{k-1}\\
&=\frac{1}{2}\frac{(n-1)!}{(2m-3)!!} \sum_{k=0}^{m}\frac{(2m-2k-1)!!}{(n-k)!} \binom{2m-k-1}{k-1} \\
&\quad\times\bigl[\bigl(1-\sqrt{z}\,\bigr)^{n-k} -(-1)^{k}\bigl(1+\sqrt{z}\,\bigr)^{n-k}\bigr]\bigl(\sqrt{z}\,\bigr)^{k-1}\\
&=\frac{1}{2^{m+1}}\frac{(m-1)!}{(2m-3)!!} \sum_{k=0}^{m}2^k\binom{n-1}{k-1}\binom{2m-k-1}{m-1}\\
&\quad\times\bigl[\bigl(1-\sqrt{z}\,\bigr)^{n-k} -(-1)^{k}\bigl(1+\sqrt{z}\,\bigr)^{n-k}\bigr]\bigl(\sqrt{z}\,\bigr)^{k-1}
\end{align*}
for $m,n\in\mathbb{N}$, where we used the immediate equalities
\begin{equation}\label{Comb-expans}
\biggl(\frac32-m\biggr)_m=\frac{(-1)^{m-1}}{2^m}(2m-3)!!
\quad\text{and}\quad
\langle n\rangle_k=\frac{n!}{(n-k)!}.
\end{equation}
Further replacing $\sqrt{z}\,$ by $z$, considering $\binom{n}{-1}=0$ for $n\in\mathbb{N}_0$, and simplifying leads to the formula~\eqref{Gauaa1-eq}.
The proof of Theorem~\ref{Gauaa1-thm} is complete.
\end{proof}

\begin{proof}[Proof of Theorem~\ref{Gauaa2-thm}]
In~\cite[pp.~1015--1016]{Gradshteyn-Ryzhik-Table-8th}, we find
\begin{align*}
{\,}_2F_1\biggl(-\frac{n}{2},\frac{1-n}{2};\frac{1}{2};\frac{z^2}{t^2}\biggr)&=\frac{(t+z)^n+(t-z)^n}{2t^{n}}
\intertext{and}
{\,}_2F_1\biggl(-\frac{n}{2},\frac{1-n}{2};\frac{1}{2};-\tan^2z\biggr)&=\frac{\cos(n z)}{\cos^{n}z}
\end{align*}
for $n\in\mathbb{N}$. These two formulas can be reformulated as
\begin{equation}\label{gauss-form2}
{\,}_2F_1\biggl(-\frac{n}{2},\frac{1-n}{2};\frac{1}{2};z\biggr)=\frac{\bigl(1+\sqrt{z}\,\bigr)^n+\bigl(1-\sqrt{z}\,\bigr)^n}{2}, \quad n\in\mathbb{N}.
\end{equation}
Accordingly, replacing $n$ by $m$ in~\eqref{drive-form-gauss-n}, taking $(a,b;c)=\bigl(-\frac{n}{2},\frac{1-n}{2};\frac{1}{2}\bigr)$ in~\eqref{drive-form-gauss-n}, and utilizing~\eqref{gauss-form2}, we obtain
\begin{align*}
{\,}_2F_1\biggl(&-\frac{n}{2},\frac{1-n}{2};\frac{1}{2}-m;z\biggr)
=\frac{z^{m+1/2}}{\bigl(\frac12-m\bigr)_m} \biggl[\frac{1}{\sqrt{z}\,}{\,}_2F_1\biggl(-\frac{n}{2},\frac{1-n}{2};\frac{1}{2};z\biggr)\biggr]^{(m)}\\
&=\frac{z^{m+1/2}}{\bigl(\frac12-m\bigr)_m} \biggl[\frac{1}{\sqrt{z}\,}\frac{\bigl(1+\sqrt{z}\,\bigr)^n+\bigl(1-\sqrt{z}\,\bigr)^n}{2}\biggr]^{(m)}\\
&=\frac{1}{2}\frac{z^{m+1/2}}{\bigl(\frac12-m\bigr)_m}\sum_{k=0}^{m}\binom{m}{k}\biggl(\frac{1}{\sqrt{z}\,}\biggr)^{(m-k)} \bigl[\bigl(1+\sqrt{z}\,\bigr)^n+\bigl(1-\sqrt{z}\,\bigr)^n\bigr]^{(k)}\\
&=\frac{1}{2}\frac{z^{m+1/2}}{\bigl(\frac12-m\bigr)_m}\sum_{k=0}^{m}\binom{m}{k}\biggl\langle-\frac{1}{2}\biggr\rangle_{m-k} z^{k-m-1/2}\\
&\quad\times\Biggl[\frac{\bigl(1+\sqrt{z}\,\bigr)^n}{(-2z)^k} \sum_{\ell=0}^{k} (-n)_\ell (2k-2\ell-1)!!\binom{2k-\ell-1}{\ell-1} \biggl(\frac{\sqrt{z}\,}{1+\sqrt{z}\,}\biggr)^\ell\\
&\quad+\frac{\bigl(1-\sqrt{z}\,\bigr)^n}{(-2z)^k} \sum_{\ell=0}^{k} \langle n\rangle_\ell(2k-2\ell-1)!!\binom{2k-\ell-1}{\ell-1} \biggl(\frac{\sqrt{z}\,}{1-\sqrt{z}\,}\biggr)^\ell\Biggr]\\
&=\frac{1}{2}\frac{1}{\bigl(\frac12-m\bigr)_m}\sum_{k=0}^{m}\binom{m}{k}\biggl\langle-\frac{1}{2}\biggr\rangle_{m-k} \frac{(-1)^k}{2^k} \sum_{\ell=0}^{k} \langle n\rangle_\ell(2k-2\ell-1)!!\\
&\quad\times\binom{2k-\ell-1}{\ell-1} \bigl(\sqrt{z}\,\bigr)^\ell \bigl[\bigl(1-\sqrt{z}\,\bigr)^{n-\ell} +(-1)^{\ell}\bigl(1+\sqrt{z}\,\bigr)^{n-\ell}\bigr]\\
&=\frac{n!}{\binom{2m}{m}} \sum_{k=0}^{m}\frac{2^{k-1}}{k!}\binom{2m-2k}{m-k} \sum_{\ell=0}^{k} \frac{(2k-2\ell-1)!!}{(n-\ell)!}\\
&\quad\times\binom{2k-\ell-1}{\ell-1} \bigl[\bigl(1-\sqrt{z}\,\bigr)^{n-\ell} +(-1)^{\ell}\bigl(1+\sqrt{z}\,\bigr)^{n-\ell}\bigr] \bigl(\sqrt{z}\,\bigr)^\ell
\end{align*}
for $m,n\in\mathbb{N}$, where we used the derivative formulas~\eqref{Plus-Deriv} and~\eqref{Minus-Deriv}, and utilized
\begin{equation}\label{Minus-half-poch}
\biggl\langle-\frac{1}{2}\biggr\rangle_{m-k}=(-1)^{m-k}\frac{(2m-2k-1)!!}{2^{m-k}}
\end{equation}
and
\begin{equation}\label{half-poch}
\biggl(\frac12-m\biggr)_m=(-1)^m\frac{(2m-1)!!}{2^m}.
\end{equation}
Further replacing $z$ by $z^2$ results in the explicit expression~\eqref{Gauaa2-eq}.
For $m=0$ the formula~\eqref{Gauaa2-eq} is just~\eqref{gauss-form2} with $z$ replaced by $z^2$, because $(-1)!!\binom{-1}{-1}=1$. The proof of Theorem~\ref{Gauaa2-thm} is complete.
\end{proof}

\begin{proof}[Proof of Theorem~\ref{comb-id-thm1}]
In~\cite[p.~556, Entry~15.1.10]{abram}, we find
\begin{equation}\label{Entry15.1.10abram}
{\,}_2F_1\biggl(a,a+\frac12;\frac32;z^2\biggr)
=\frac{(1+z)^{1-2a}-(1-z)^{1-2a}}{2(1-2a)z}.
\end{equation}
Replacing $n$ with $m$ in~\eqref{drive-form-gauss-n}, then setting $(a,b;c)=\bigl(a,a+\frac{1}{2};\frac{3}{2}\bigr)$ in~\eqref{drive-form-gauss-n}, and finally rearranging the resulting expression, we obtain
\begin{multline}\label{Gauss-m-Eq}
{\,}_2F_1\biggl(a,a+\frac{1}{2};\frac{3}{2}-m;z\biggr)
=\frac{z^{m-1/2}}{\bigl(\frac{3}{2}-m\bigr)_m} \biggl[\sqrt{z}\,{\,}_2F_1\biggl(a,a+\frac{1}{2};\frac{3}{2};z\biggr)\biggr]^{(m)}\\
=(-1)^{m-1}\frac{2^mz^{m-1/2}}{(2m-3)!!} \biggl[\frac{\bigl(1+\sqrt{z}\,\bigr)^{1-2a}-\bigl(1-\sqrt{z}\,\bigr)^{1-2a}}{2(1-2a)}\biggr]^{(m)}\\
=\frac{(-1)^{m-1}}{1-2a} \frac{2^{m-1}z^{m-1/2}}{(2m-3)!!} \bigl[\bigl(1+\sqrt{z}\,\bigr)^{1-2a}-\bigl(1-\sqrt{z}\,\bigr)^{1-2a}\bigr]^{(m)}
\end{multline}
for $m\in\mathbb{N}_0$, where we used~\eqref{Entry15.1.10abram} and the first one in~\eqref{Comb-expans}.
\par
Taking $\alpha=1$, $\beta=\pm1$, and $\gamma=1-2a$ in Lemma~\ref{lem-deriv-Gauss} yields
\begin{multline*}
\bigl[\bigl(1+\sqrt{z}\,\bigr)^{1-2a}\bigr]^{(m)}\\
=\frac{(-1)^m}{(2z)^m} \sum_{k=0}^{m}(2a-1)_k (2m-2k-1)!!\binom{2m-k-1}{k-1}\bigl(1+\sqrt{z}\,\bigr)^{1-2a-k}\bigl(\sqrt{z}\,\bigr)^k
\end{multline*}
and
\begin{multline*}
\bigl[\bigl(1-\sqrt{z}\,\bigr)^{1-2a}\bigr]^{(m)}\\
=\frac{(-1)^m}{(2z)^m} \sum_{k=0}^{m} \langle1-2a\rangle_k(2m-2k-1)!!\binom{2m-k-1}{k-1} \bigl(1-\sqrt{z}\,\bigr)^{1-2a-k} \bigl(\sqrt{z}\,\bigr)^k
\end{multline*}
for $m\in\mathbb{N}_0$. As a result, we acquire
\begin{multline*}
\bigl[\bigl(1+\sqrt{z}\,\bigr)^{1-2a}-\bigl(1-\sqrt{z}\,\bigr)^{1-2a}\bigr]^{(m)}
=\frac{(-1)^m}{(2z)^m} \sum_{k=0}^{m}(2m-2k-1)!!\binom{2m-k-1}{k-1}\\
\times\bigl[(2a-1)_k\bigl(1+\sqrt{z}\,\bigr)^{1-2a-k} -\langle1-2a\rangle_k\bigl(1-\sqrt{z}\,\bigr)^{1-2a-k}\bigr]\bigl(\sqrt{z}\,\bigr)^k.
\end{multline*}
Substituting this result into~\eqref{Gauss-m-Eq} leads to
\begin{multline*}
{\,}_2F_1\biggl(a,a+\frac{1}{2};\frac{3}{2}-m;z\biggr)
=\frac{1}{2(2a-1)} \frac{1}{(2m-3)!!} \sum_{k=0}^{m}(2m-2k-1)!!\binom{2m-k-1}{k-1}\\
\times\bigl[(2a-1)_k\bigl(1+\sqrt{z}\,\bigr)^{1-2a-k} -\langle1-2a\rangle_k\bigl(1-\sqrt{z}\,\bigr)^{1-2a-k}\bigr]\bigl(\sqrt{z}\,\bigr)^{k-1}.
\end{multline*}
Further replacing $\sqrt{z}\,$ by $z$ yields~\eqref{comb-id-Eq1}. The proof of Theorem~\ref{comb-id-thm1} is complete.
\end{proof}

\begin{proof}[Proof of Theorem~\ref{comb-id-thm2}]
In~\cite[p.~556, Entry~15.1.9]{abram}, we find
\begin{equation}\label{Entry15.1.9abram}
{\,}_2F_1\biggl(a,a+\frac12;\frac12;z^2\biggr)
=\frac{(1+z)^{-2a}+(1-z)^{-2a}}{2}.
\end{equation}
Replacing $n$ with $m$ in~\eqref{drive-form-gauss-n}, then letting $(a,b;c)=\bigl(a,a+\frac{1}{2};\frac{1}{2}\bigr)$ in~\eqref{drive-form-gauss-n}, and then using~\eqref{Entry15.1.9abram}, we obtain
\begin{align*}
&{\,}_2F_1\biggl(a,a+\frac{1}{2};\frac{1}{2}-m;z\biggr)
=\frac{z^{m+1/2}}{\bigl(\frac{1}{2}-m\bigr)_m} \biggl[\frac{{\,}_2F_1\bigl(a,a+\frac{1}{2};\frac{1}{2};z\bigr)}{\sqrt{z}\,}\biggr]^{(m)}\\
&=(-1)^{m}\frac{2^mz^{m+1/2}}{(2m-1)!!} \biggl[\frac{\bigl(1+\sqrt{z}\,\bigr)^{-2a}+\bigl(1-\sqrt{z}\,\bigr)^{-2a}}{2\sqrt{z}\,}\biggr]^{(m)}\\
&=(-1)^{m}\frac{2^{m-1}z^{m+1/2}}{(2m-1)!!} \sum_{k=0}^m\binom{m}{k}\biggl(\frac{1}{\sqrt{z}\,}\biggr)^{(m-k)} \bigl[\bigl(1+\sqrt{z}\,\bigr)^{-2a}+\bigl(1-\sqrt{z}\,\bigr)^{-2a}\bigr]^{(k)}\\
&=\frac{1}{2^{m+1}}\frac{m!}{(2m-1)!!} \sum_{k=0}^m\frac{2^k}{k!}\binom{2m-2k}{m-k} \sum_{\ell=0}^{k} (2k-2\ell-1)!!\binom{2k-\ell-1}{\ell-1}\\
&\quad\times \bigl[(2a)_\ell\bigl(1+\sqrt{z}\,\bigr)^{-2a-\ell}+\langle-2a\rangle_\ell \bigl(1-\sqrt{z}\,\bigr)^{-2a-\ell}\bigr]\bigl(\sqrt{z}\,\bigr)^\ell,
\end{align*}
where we used the identities~\eqref{Minus-half-poch} and~\eqref{half-poch} and employed the derivative formulas
\begin{multline*}
\bigl[\bigl(1+\sqrt{z}\,\bigr)^{-2a}\bigr]^{(k)}\\
=\frac{(-1)^k}{(2z)^k} \sum_{\ell=0}^{k}(2a)_\ell (2k-2\ell-1)!!\binom{2k-\ell-1}{\ell-1}\bigl(1+\sqrt{z}\,\bigr)^{-2a-\ell}\bigl(\sqrt{z}\,\bigr)^\ell
\end{multline*}
and
\begin{multline*}
\bigl[\bigl(1-\sqrt{z}\,\bigr)^{-2a}\bigr]^{(k)}\\
=\frac{(-1)^k}{(2z)^k} \sum_{\ell=0}^{k} \langle-2a\rangle_\ell(2k-2\ell-1)!!\binom{2k-\ell-1}{\ell-1} \bigl(1-\sqrt{z}\,\bigr)^{-2a-\ell} \bigl(\sqrt{z}\,\bigr)^\ell
\end{multline*}
for $k\in\mathbb{N}_0$, which are deduced by taking $\alpha=1$, $\beta=\pm1$, and $\gamma=-2a$ in~\eqref{Eq-deriv-Gauss}.
Further replacing $\sqrt{z}\,$ by $z$ yields~\eqref{comb-id-Eq2}. The proof of Theorem~\ref{comb-id-thm2} is complete.
\end{proof}

\section{Combinatorial identities}\label{sec-combin-IDs}
In this section, using the explicit formulas~\eqref{Gauaa1-eq} and~\eqref{Gauaa2-eq} from Theorems~\ref{Gauaa1-thm} and~\ref{Gauaa2-thm}, together with~\eqref{comb-id-Eq1} and~\eqref{comb-id-Eq2} from Theorems~\ref{comb-id-thm1} and~\ref{comb-id-thm2}, we obtain the following combinatorial identities.

\begin{thm}\label{oKpopThm}
For $n,m\in\mathbb{N}_0$ with $n\ge m$, we have the combinatorial identity
\begin{equation}\label{oKpop}
\sum_{k=0}^{m}\frac{2^k}{k!} \binom{2m-2k}{m-k} \sum_{\ell=0}^{k} \frac{(-1)^\ell}{2^\ell} \frac{(2k-2\ell-1)!!}{(n-\ell)!} \binom{2k-\ell-1}{\ell-1}
=\frac{1}{n!}\binom{2m-n}{m}.
\end{equation}
\end{thm}

\begin{proof}
Replacing $m$ by $m-1$ and $n$ by $n-1$ in~\eqref{Gauaa2-eq} gives
\begin{multline*}
{\,}_2F_1\biggl(\frac{1-n}{2},\frac{2-n}{2};\frac{3}{2}-m;z^2\biggr)
=\frac{(n-1)!}{\binom{2(m-1)}{m-1}} \sum_{k=0}^{m-1}\frac{2^{k-1}}{k!}\binom{2m-2k-2}{m-k-1}\\
\times\sum_{\ell=0}^{k} \frac{(2k-2\ell-1)!!}{(n-\ell-1)!} \binom{2k-\ell-1}{\ell-1} \bigl[(1-z)^{n-\ell-1} +(-1)^{\ell}(1+z)^{n-\ell-1}\bigr] z^\ell
\end{multline*}
for $m,n\in\mathbb{N}$.
Comparing this with~\eqref{Gauaa1-eq} in Theorem~\ref{Gauaa1-thm} and simplifying results in
\begin{multline*}
\sum_{k=0}^{m-1}(2m-2k-3)!!\binom{m-1}{k} \sum_{\ell=0}^{k} \frac{(2k-2\ell-1)!!}{(n-\ell-1)!} \binom{2k-\ell-1}{\ell-1} \\
\times\bigl[(1-z)^{n-\ell-1} +(-1)^{\ell}(1+z)^{n-\ell-1}\bigr] z^\ell\\
=\frac{(m-1)!}{2^m(n-1)!}\sum_{k=1}^{m}2^k\binom{n-1}{k-1} \binom{2m-k-1}{m-1}\bigl[(1-z)^{n-k}-(-1)^{k}(1+z)^{n-k}\bigr]z^{k-1}.
\end{multline*}
When $n,m\in\mathbb{N}$ with $n>m$, further letting $z\to\pm1$ and simplifying lead to
\begin{multline}\label{opopopop}
\sum_{k=0}^{m-1}(2m-2k-3)!!\binom{m-1}{k} \sum_{\ell=0}^{k} \frac{(-1)^\ell}{2^\ell} \frac{(2k-2\ell-1)!!}{(n-\ell-1)!} \binom{2k-\ell-1}{\ell-1}\\
=\frac{(m-1)!}{2^{m-1}(n-1)!}\sum_{k=1}^{m}(-1)^{k-1} \binom{n-1}{k-1} \binom{2m-k-1}{m-1}\\
=\frac{(m-1)!}{2^{m-1}(n-1)!}\sum_{k=0}^{m-1}(-1)^{k} \binom{n-1}{k} \binom{2m-k-2}{m-1}.
\end{multline}
In~\cite[p.~65]{Sprugnoli-Gould-2006}, we find the combinatorial identity
\begin{equation*}
\sum_{k=0}^{n}(-1)^k\binom{x}{k}\binom{2n-k}{n}=\binom{2n-x}{n}.
\end{equation*}
This implies that
\begin{equation*}
\sum_{k=0}^{m-1}(-1)^{k} \binom{n-1}{k} \binom{2m-k-2}{m-1}=\binom{2m-n-1}{m-1}.
\end{equation*}
Substituting this identity into~\eqref{opopopop} yields
\begin{multline*}
\sum_{k=0}^{m-1}(2m-2k-3)!!\binom{m-1}{k} \sum_{\ell=0}^{k} \frac{(-1)^\ell}{2^\ell} \frac{(2k-2\ell-1)!!}{(n-\ell-1)!} \binom{2k-\ell-1}{\ell-1}\\
=\frac{(m-1)!}{2^{m-1}(n-1)!}\binom{2m-n-1}{m-1},
\end{multline*}
which can be rearranged as
\begin{multline}\label{opop}
\sum_{k=0}^{m-1}\frac{2^k}{k!} \binom{2m-2k-2}{m-k-1} \sum_{\ell=0}^{k} \frac{(-1)^\ell}{2^\ell} \frac{(2k-2\ell-1)!!}{(n-\ell-1)!} \binom{2k-\ell-1}{\ell-1}\\
=\frac{1}{(n-1)!}\binom{2m-n-1}{m-1},
\end{multline}
for $n,m\in\mathbb{N}$ with $n>m$.
\par
On~\cite[p.~225]{DA19034-cas-sc.tex}, in the first proof of~\cite[Theorem~4]{DA19034-cas-sc.tex}, we obtained
\begin{multline*}
\sum_{k=0}^{n-1} \frac{2^{k}}{k!}\binom{2n-2k-2}{n-k-1} \sum_{j=0}^{k}\frac{(-1)^{j}}{2^{j}}\frac{(2k-2j-1)!!}{(n-j-1)!}\binom{2k-j-1}{j-1}\\
=\frac{1}{(n-1)!}\sum_{k=1}^{n}(-1)^{k-1}\binom{2n-k-1}{n-1}\binom{n-1}{k-1}
=\frac{1}{(n-1)!}
\end{multline*}
for $n\in\mathbb{N}$. Combining this identity with~\eqref{opop} for $n,m\in\mathbb{N}$ with $n>m$ reveals that the identity~\eqref{opop} is valid for $n,m\in\mathbb{N}$ with $n\ge m$.
\par
Replacing $m$ by $m+1$ and $n$ by $n+1$ in~\eqref{opop} for $n,m\in\mathbb{N}$ with $n\ge m$ leads to~\eqref{oKpop} for $n,m\in\mathbb{N}_0$ with $n\ge m$. The proof of Theorem~\ref{oKpopThm} is complete.
\end{proof}

\begin{rem}\label{Rem-Conj}
We conjecture that the combinatorial identity~\eqref{oKpop} in Theorem~\ref{oKpopThm} remains valid for all $n,m\in\mathbb{N}_0$ with $n<m$. Equivalently, the identity holds for every $m\in\mathbb{N}_0$ and $n\in\mathbb{Z}$. See \url{https://mathoverflow.net/q/513017} (accessed on 7 July 2026).
\end{rem}

\begin{thm}\label{ThmoKpopThm}
For $m\in\mathbb{N}_0$ and $a\in\mathbb{C}$, we have the combinatorial identity
\begin{equation}\label{ThmoKpopQe}
\sum_{k=1}^{m}\frac{1}{(k!)^2}\binom{2m-2k}{m-k}\sum_{\ell=1}^{k} \binom{k}{\ell}\ell(2k-\ell-1)! (2a)_\ell
=\binom{2m+2a}{m}-\binom{2m}{m}.
\end{equation}
\end{thm}

\begin{proof}
Replacing $m\in\mathbb{N}$ by $m-1$ in~\eqref{comb-id-Eq2} and using
\begin{equation*}
\frac{1}{2^{m}}\frac{(m-1)!}{(2m-3)!!}\frac{2^k}{k!}\binom{2m-2k-2}{m-k-1}
=\frac{1}{2}\frac{1}{(2m-3)!!}(2m-2k-3)!!\binom{m-1}{k}
\end{equation*}
gives
\begin{multline*}
{\,}_2F_1\biggl(a,a+\frac{1}{2};\frac{3}{2}-m;z^2\biggr)
=\frac{1}{2}\frac{1}{(2m-3)!!} \sum_{k=0}^{m-1}(2m-2k-3)!!\binom{m-1}{k} \\
\times \sum_{\ell=0}^{k} (2k-2\ell-1)!!\binom{2k-\ell-1}{\ell-1}\bigl[(2a)_\ell(1+z)^{-2a-\ell}+\langle-2a\rangle_\ell (1-z)^{-2a-\ell}\bigr]z^\ell.
\end{multline*}
For $k\in\mathbb{N}$ the term $\ell=0$ vanishes because $\binom{2k-1}{-1}=0$, while for $k=\ell=0$ the term equals $(-1)!!\binom{-1}{-1}=1$. For $\ell\in\mathbb{N}$ we have
\begin{equation*}
(2k-2\ell-1)!!\binom{2k-\ell-1}{\ell-1}
=\frac{(2k-\ell-1)!}{2^{k-\ell}(k-\ell)!\,(\ell-1)!}
=\frac{2^\ell\ell(2k-\ell-1)!}{k!\,2^{k}}\binom{k}{\ell}.
\end{equation*}
Consequently,
\begin{multline*}
{\,}_2F_1\biggl(a,a+\frac{1}{2};\frac{3}{2}-m;z^2\biggr)
=\frac{1}{2}\bigl[(1+z)^{-2a}+(1-z)^{-2a}\bigr]\\
+\frac{1}{2}\frac{1}{(2m-3)!!} \sum_{k=1}^{m-1}(2m-2k-3)!!\binom{m-1}{k}\\
\times\sum_{\ell=1}^{k} \frac{2^\ell\ell(2k-\ell-1)!}{k!2^{k}}\binom{k}{\ell}\bigl[(2a)_\ell(1+z)^{-2a-\ell}+\langle-2a\rangle_\ell (1-z)^{-2a-\ell}\bigr]z^\ell.
\end{multline*}
Comparing this with~\eqref{comb-id-Eq1} and simplifying result in
\begin{multline*}
(2m-3)!!\bigl[(1+z)^{-2a}+(1-z)^{-2a}\bigr]\\
+\sum_{k=1}^{m-1}(2m-2k-3)!!\binom{m-1}{k}\sum_{\ell=1}^{k} \frac{2^\ell\ell(2k-\ell-1)!}{k!2^{k}}\binom{k}{\ell}\\
\times\bigl[(2a)_\ell(1+z)^{-2a-\ell}+\langle-2a\rangle_\ell (1-z)^{-2a-\ell}\bigr]z^\ell\\
=\frac{1}{2a-1} \sum_{k=0}^{m}(2m-2k-1)!!\binom{2m-k-1}{k-1}\\
\times\bigl[(2a-1)_k(1+z)^{1-2a-k} -\langle1-2a\rangle_k(1-z)^{1-2a-k}\bigr]z^{k-1}
\end{multline*}
for $m\in\mathbb{N}$ and $a\ne\frac12$. For $m\in\mathbb{N}$ and $a<\frac{1-m}{2}$, letting $z\to1$ and dividing by $2^{-2a}$ gives
\begin{multline}\label{a-final-Id}
(2m-3)!!+\sum_{k=1}^{m-1}(2m-2k-3)!!\binom{m-1}{k}\sum_{\ell=1}^{k} \frac{2^\ell\ell(2k-\ell-1)!}{k!2^{k}}\binom{k}{\ell}\frac{(2a)_\ell}{2^\ell}\\
=\frac{1}{2a-1}\sum_{k=0}^{m}(2m-2k-1)!!\binom{2m-k-1}{k-1}\frac{(2a-1)_{k}}{2^{k-1}}.
\end{multline}
Substituting~\eqref{ID-Lem2} in Lemma~\ref{Lem2-ID} into~\eqref{a-final-Id} yields
\begin{multline*}
\sum_{k=1}^{m-1}(2m-2k-3)!!\binom{m-1}{k}\sum_{\ell=1}^{k} \frac{2^\ell\ell(2k-\ell-1)!}{k!2^{k}}\binom{k}{\ell}\frac{(2a)_\ell}{2^\ell}\\
=\frac{1}{2^{m-1}}\frac{\Gamma(2a+2m-1)}{\Gamma(2a+m)}-(2m-3)!!
\end{multline*}
for $m\in\mathbb{N}$ and $a<\frac{1-m}{2}$. Further replacing $m$ by $m+1$ leads to
\begin{multline*}
\frac{m!}{2^m}\sum_{k=1}^{m}\frac{2^k}{k!}\binom{2m-2k}{m-k}\sum_{\ell=1}^{k} \frac{2^\ell\ell(2k-\ell-1)!}{k!2^{k}}\binom{k}{\ell}\frac{(2a)_\ell}{2^\ell}\\
=\frac{1}{2^{m}}\frac{\Gamma(2a+2m+1)}{\Gamma(2a+m+1)}-(2m-1)!!
\end{multline*}
for $m\in\mathbb{N}_0$ and $a<-\frac{m}{2}$. Multiplying by $\frac{2^m}{m!}$ and using 
$$
\frac{\Gamma(2a+2m+1)}{\Gamma(2a+m+1)}=m!\binom{2m+2a}{m}\quad \text{and}\quad \frac{2^m(2m-1)!!}{m!}=\binom{2m}{m},
$$ 
this becomes the combinatorial identity~\eqref{ThmoKpopQe} for $m\in\mathbb{N}_0$ and $a<-\frac{m}{2}$.
\par
Since both sides of~\eqref{ThmoKpopQe} are polynomials in $a$ for fixed $m\in\mathbb{N}_0$, and they agree for all $a<-\frac{m}{2}$, the identity~\eqref{ThmoKpopQe} holds for all $a\in\mathbb{C}$. The proof of Theorem~\ref{ThmoKpopThm} is complete.
\end{proof}

\begin{cor}\label{cor-Gauss-2F}
The identity~\eqref{oKpop} in Theorem~\ref{oKpopThm} is valid for all $m\in\mathbb{N}_0$ and $n\in\mathbb{Z}$. Equivalently, the conjecture posed in Remark~\ref{Rem-Conj} is true.
\end{cor}

\begin{proof}
Let $m\in\mathbb{N}$. For $n\in\mathbb{N}_0$, taking $a=-\frac{n}{2}$ in the identity~\eqref{ThmoKpopQe} of Theorem~\ref{ThmoKpopThm} yields
\begin{equation}\label{a=-n/2}
\sum_{k=1}^{m}\frac{1}{(k!)^2}\binom{2m-2k}{m-k}\sum_{\ell=1}^{k} \binom{k}{\ell}\ell(2k-\ell-1)! (-n)_\ell
=\binom{2m-n}{m}-\binom{2m}{m}.
\end{equation}
Straightforward computation shows that
\begin{align*}
\binom{k}{\ell}\ell(2k-\ell-1)! (-n)_\ell
&=\frac{k!}{\ell!(k-\ell)!}\ell(2k-\ell-1)! (-1)^\ell\langle n\rangle_\ell\\
&=(-1)^\ell\frac{k!}{\ell!(k-\ell)!}(2k-2\ell)! \ell!\frac{(2k-\ell-1)!}{(2k-2\ell)!(\ell-1)!}\frac{n!}{(n-\ell)!}\\
&=(-1)^\ell\frac{k!n!}{(k-\ell)!}(2k-2\ell)!! \frac{(2k-2\ell-1)!!}{(n-\ell)!}\binom{2k-\ell-1}{\ell-1}\\
&=(2k)!!n! \frac{(-1)^\ell}{2^\ell}\frac{(2k-2\ell-1)!!}{(n-\ell)!}\binom{2k-\ell-1}{\ell-1}
\end{align*}
for $1\le\ell\le k$ and $n\in\mathbb{N}_0$, where $\frac{1}{(n-\ell)!}=0$ if $\ell>n$. Substituting this into~\eqref{a=-n/2} leads to
\begin{equation*}
n!\sum_{k=1}^{m}\frac{2^k}{k!} \binom{2m-2k}{m-k} \sum_{\ell=1}^{k} \frac{(-1)^\ell}{2^\ell} \frac{(2k-2\ell-1)!!}{(n-\ell)!} \binom{2k-\ell-1}{\ell-1}=\binom{2m-n}{m}-\binom{2m}{m}
\end{equation*}
for $n\in\mathbb{N}_0$. Dividing both sides by $n!$ and adding the term $k=\ell=0$ of~\eqref{oKpop}, which equals $\frac{1}{n!}\binom{2m}{m}$ (the terms with $k\ge1$ and $\ell=0$ vanish since $\binom{2k-1}{-1}=0$), yields the identity~\eqref{oKpop} for $m\in\mathbb{N}$ and $n\in\mathbb{N}_0$. If $n$ is a negative integer, then $\frac1{(n-\ell)!}=0$ for all $\ell\ge0$ and $\frac1{n!}\binom{2m-n}{m}=0$, so both sides of~\eqref{oKpop} vanish.
\par
When $m=0$, the identity~\eqref{oKpop} becomes
\begin{equation*}
\frac{(-1)!!}{n!} \binom{-1}{-1}=\frac{1}{n!}\binom{-n}{0},
\end{equation*}
whose validity is trivial. The proof of Corollary~\ref{cor-Gauss-2F} is complete.
\end{proof}

\section{Conclusions}
The main results and highlights of this paper include the explicit formula~\eqref{comb-id-Eq1} in Theorem~\ref{comb-id-thm1}, the explicit formula~\eqref{comb-id-Eq2} in Theorem~\ref{comb-id-thm2}, and the combinatorial identity~\eqref{ThmoKpopQe} in Theorem~\ref{ThmoKpopThm}. 
\begin{enumerate}
\item
From the explicit formula~\eqref{comb-id-Eq1} in Theorem~\ref{comb-id-thm1}, we can obtain the explicit formula~\eqref{2F1(izan-peraz)} in Theorem~A (\cite[Theorem~1]{DA19034-cas-sc.tex}) as well as the explicit formula~\eqref{Gauaa1-eq} in Theorem~\ref{Gauaa1-thm}. 
\item
From the explicit formula~\eqref{comb-id-Eq2} in Theorem~\ref{comb-id-thm2}, we can derive the explicit formula~\eqref{2F1(izan-peraz)-more} in Theorem~B (\cite[Theorem~2]{DA19034-cas-sc.tex}) together with the explicit formula~\eqref{Gauaa2-eq} in Theorem~\ref{Gauaa2-thm}. 
\item
Finally, from the combinatorial identity~\eqref{ThmoKpopQe} in Theorem~\ref{ThmoKpopThm}, we extend the combinatorial identity~\eqref{oKpop} in Theorem~\ref{oKpopThm} and the identity~\eqref{ID-n-Factorial2sums} in Theorem~D (\cite[Theorem~4]{DA19034-cas-sc.tex}).
\end{enumerate}
\par
The idea, method, and approach presented in this paper can be applied to derive more combinatorial identities and to investigate the Legendre polynomials $P_n(z)$, the Jacobi polynomials $P_n^{(\rho,\sigma)}(z)$, the Gegenbauer (or ultraspherical) polynomials $C_n^\nu(z)$, and their corresponding functions.
\par
This paper is a revised version of the arXiv preprint arXiv:2607.10643 at \url{https://doi.org/10.48550/arXiv.2607.10643}.

\section{Declarations}

\paragraph*{\bf Author's Contributions}
The author contributed to the whole manuscript and read and approved the final manuscript.

\paragraph*{\bf Funding}
The author was partially supported by the Natural Science Foundation of Inner Mongolia Autonomous Region (Grant No.~2025QN01041) and by the Youth Project of Hulunbuir City for Basic Research and Applied Basic Research (Grant No.~GH2024020).

\paragraph*{\bf Institutional Review Board Statement}
Not applicable.

\paragraph*{\bf Informed Consent Statement}
Not applicable.

\paragraph*{\bf Ethical Approval}
The conducted research does not involve human or animal subjects.

\paragraph*{\bf Availability of Data and Material}
Data sharing is not applicable to this article as no new data were created or analyzed in this study.

\paragraph*{\bf Acknowledgements}
Not applicable.

\paragraph*{\bf Competing Interests}
The author declares no competing interests.

\paragraph*{\bf Use of AI Tools Declaration}
The author declares that no AI tools were used in the creation of this article.

\end{document}